\documentclass[12pt]{amsart}
\usepackage{ifthen,verbatim}
\usepackage{floatflt,graphicx,graphics}

 \usepackage{tikz} 
\usepackage{mathrsfs}
\usepackage{amsmath,amssymb,amsthm}
\usepackage{mathtools}

\numberwithin{equation}{section}
\nonstopmode
\numberwithin{equation}{section}
\theoremstyle{plain}
\newtheorem{theorem}[equation]{Theorem}

\newtheorem{lemma}[equation]{Lemma}
\newtheorem{corollary}[equation]{Corollary}

\theoremstyle{definition}

\newtheorem{example}[equation]{Example}
\newtheorem{remark}[equation]{Remark}
\newtheorem{nonsec}[equation]{}

\theoremstyle{remark}

\newcommand{\R}{\mathbb{R}}

\newcommand{\C}{\mathbb{C}}

\newcommand{\B}{\mathbb{B}}
\newcommand{\M}{\mathsf{M}}
\newcommand{\K}{\mathcal{K}}
\newcommand{\uhp}{\mathbb{H}}

\font\fFt=eusm10 
\font\fFa=eusm7  
\font\fFp=eusm5  
\def\K{\mathchoice
{\hbox{\,\fFt K}}
{\hbox{\,\fFt K}}
{\hbox{\,\fFa K}}
{\hbox{\,\fFp K}}}

{\qed\bigskip}

\newcounter{alphabet}

\newcounter{minutes}
\divide\time by 60
\newcounter{hours}
\multiply\time by 60
\addtocounter{minutes}{-\time}

\begin{document}
\bibliographystyle{amsplain}
\title
{
Intrinsic metrics under conformal and quasiregular mappings
}

\def\thefootnote{}
\footnotetext{
\texttt{\tiny File:~\jobname .tex,
          printed: \number\year-\number\month-\number\day,
          \thehours.\ifnum\theminutes<10{0}\fi\theminutes}
}
\makeatletter\def\thefootnote{\@arabic\c@footnote}\makeatother

\author[O. Rainio]{Oona Rainio}
\address{Department of Mathematics and Statistics, University of Turku, FI-20014 Turku, Finland}
\email{ormrai@utu.fi}

\begin{abstract}
The distortion of six different intrinsic metrics and quasi-metrics under conformal and quasiregular mappings is studied in a few simple domains $G\subsetneq\mathbb{R}^n$. The already known inequalities between the hyperbolic metric and these intrinsic metrics for points $x,y$ in the unit ball $\mathbb{B}^n$ are improved by limiting the absolute values of the points $x,y$ and the new results are then used to study the conformal distortion of the intrinsic metrics. For the triangular ratio metric between two points $x,y\in\mathbb{B}^n$, the conformal distortion is bounded in terms of the hyperbolic midpoint and the hyperbolic distance of $x,y$. Furthermore, quasiregular and quasiconformal mappings are studied, and new sharp versions of the Schwarz lemma are introduced.
\end{abstract}
\maketitle

\textbf{Conflict of interest statement.} On behalf of all authors, the corresponding author states that there is no conflict of interest.

\textbf{Data availibility statement.} Data sharing not applicable – no new data generated.

\textbf{Funding.} The author has a PhD training position funded by the University of Turku Graduate School UTUGS.

{\bf Acknowledgements.} This research continues my earlier work in \cite{fss, inm, sinb, sqm}. Three last mentioned of these articles have been co-written with Professor Matti Vuorinen, to whom I am indebted for all guidance and other support. 

\section{Introduction}

Geometric function theory is an important field of study focusing on boundary structures of domains, several types of mappings and geometric properties of different functions. To study these concepts, one needs to find such a way to measure distances between points in a domain that illustrates the intricate features of the geometric entities so that they can be properly observed. One key concept here is an \emph{intrinsic distance}, which measures the distance between two points by taking into account how these points are located with respect to the boundary of the domain instead of just expressing how close the points are to each other.

A most well-known metric for measuring intrinsic distances is the \emph{hyperbolic metric} \cite{bm}, but it is not the only one. Instead, there are numerous distance functions with similar properties called the \emph{hyperbolic type metrics} or, more generally, the \emph{intrinsic metrics}. Recently, these metrics have been studied, for instance, in \cite{chkv, fhmv, fmv, hvz, imsz, fss, inm, sinb, sqm}.

Just like the hyperbolic metric, the intrinsic metrics are often monotonic with respect to the size of the domain, sensitive to the changes in its boundary and give increasing values for distances between two distinct points whenever the point closer to the boundary approaches it \cite[pp. 191-192]{hkvbook}. However, in spite of all these similarities, the intrinsic metrics do not generally share the most important property of the hyperbolic metric, the \emph{conformal invariance}. Namely, while the distances measured with the hyperbolic metric are always preserved under conformal mappings, the values of the intrinsic metrics generally change under these types of mappings. 

Thus, this leaves us with the question about how these metrics are exactly transformed under conformal mappings. Studying the \emph{distortion} of different metrics under conformal mappings and other well-known classes of mappings is a significant focus of study in this field. However, while the distortion of the hyperbolic metric under the quasiconformal and quasiregular mappings has been already studied, see for instance \cite[Thm 16.2, p. 300]{hkvbook}, there is not yet many results about how the newer intrinsic metrics and quasi-metrics behave under any of these mappings. 

Consequently, to fill this gap, the aim of this article is to find more about the distortion of six chosen metrics and quasi-metrics. While the focus is mostly their behaviour under conformal mappings and M\"obius transformations, also the quasiregular and quasiconformal mappings are introduced and studied. Furthermore, in order to make the found results understandable to the reader, all the definitions of different classes of mappings, the metrics considered and other related concepts are given and explained.

The structure of this paper is as follows. First, in Section \ref{s3}, we study the inequalities between the hyperbolic metric and several intrinsic metrics in the unit ball and give new sharp bounds for these metrics depending on the absolute values of the points whose intrinsic distances are being measured. Then, we use these inequalities to give bounds for the distortion under conformal mappings defined on the unit disk in Section \ref{s4}. We also consider one interesting M\"obius transformation in particular and use a so-called hyperbolic midpoint rotation to bound the conformal distortion of the triangular ratio metric. Our argument here is based on the idea of the hyperbolic midpoint rotation from \cite{sinb} and we use the very recent explicit formula for the hyperbolic midpoint from \cite[Thm 1.4, p. 3]{wvz}. Finally, in Section \ref{s5}, we apply the Schwarz lemma for quasiregular and quasiconformal mappings to prove several new H\"older continuity results for the metrics studied. These results, combined with the inequalities from the earlier sections, give also new inequalities for the distortion of quasiregular maps in the Euclidean metric.

\section{Preliminaries}

Suppose that $G$ is a proper domain in $\R^n$. In other words, choose an open, non-empty and connected set $G\subsetneq\R^n$. For any point $x\in G$, denote the Euclidean distance from it to the boundary $\partial G$ by $d_G(x)=\inf\{|x-z|\text{ }|\text{ }z\in\partial G\}$. Furthermore, for all $x\in\R^n$ and $r>0$, let $B^n(x,r)$ be the $x$-centered Euclidean open ball with the radius $r$, $\overline{B}^n(x,r)$ its closure and $S^{n-1}(x,r)$ its boundary sphere. Denote the Euclidean line passing through two distinct points $x,y\in\R^n$ by $L(x,y)$ and the Euclidean line segment with these end points by $[x,y]$.

In this paper, we often let the domain $G$ be either the upper half-space $\uhp^n=\{(x_1,...,x_n)\in\R^n\text{ }|\text{ }x_n>0\}$, the unit ball $\B^n=B^n(0,1)$ or the open sector $S_\theta=\{x\in\C\text{ }|\text{ }0<\arg(x)<\theta\}$ with an angle $\theta\in(0,2\pi)$. Define the \emph{hyperbolic metric} in these cases as 
\begin{align*}
\text{ch}\rho_{\uhp^n}(x,y)&=1+\frac{|x-y|^2}{2d_{\uhp^n}(x)d_{\uhp^n}(y)},\quad x,y\in\uhp^n,\\
\text{sh}^2\frac{\rho_{\B^n}(x,y)}{2}&=\frac{|x-y|^2}{(1-|x|^2)(1-|y|^2)},\quad x,y\in\B^n,\\
\rho_{S_\theta}(x,y)&=\rho_{\uhp^2}(x^{\pi\slash\theta},y^{\pi\slash\theta}),\quad x,y\in S_\theta.
\end{align*}
The first two formulas above are in \cite[(4.8), p. 52 \& (4.14), p. 55]{hkvbook} and the third one follows from the \emph{conformal invariance} of the hyperbolic metric: If $G$ is a domain and $f:G\to G'=f(G)$ is a conformal mapping, then
\begin{align*}
\rho_G(x,y)=\rho_{G'}(f(x),f(y))    
\end{align*}
for all points $x,y\in G$. Note also that if $n=2$, the formulas of the hyperbolic metric can be simplified to
\begin{align*}
\text{th}\frac{\rho_{\uhp^2}(x,y)}{2}=\left|\frac{x-y}{x-\overline{y}}\right|,\quad
\text{th}\frac{\rho_{\B^2}(x,y)}{2}=\left|\frac{x-y}{1-x\overline{y}}\right|=\frac{|x-y|}{A[x,y]},
\end{align*}
where $\overline{y}$ is the complex conjugate of $y$ and $A[x,y]=\sqrt{|x-y|^2+(1-|x|^2)(1-|y|^2)}$ is the Ahlfors bracket, see \cite[(3.17) p. 39]{hkvbook}.

Other than the hyperbolic metric, we will also need the following intrinsic metrics and quasi-metrics for a domain $G\subsetneq\R^n$:
The \emph{distance ratio metric}, introduced by Gehring and Palka \cite{gp}, $j_G:G\times G\to[0,\infty)$,
\begin{align*}
j_G(x,y)=\log\left(1+\frac{|x-y|}{\min\{d_G(x),d_G(y)\}}\right),   
\end{align*}
the \emph{$j^*$-metric} \cite[2.2, p. 1123 \& Lemma 2.1, p. 1124]{hvz} $j^*_G:G\times G\to[0,1],$
\begin{align*}
j^*_G(x,y)={\rm th}\frac{j_G(x,y)}{2}=\frac{|x-y|}{|x-y|+2\min\{d_G(x),d_G(y)\}},    
\end{align*}
the \emph{triangular ratio metric} \cite[(1.1), p. 683]{chkv}, originally introduced by P. H\"ast\"o in 2002 \cite{h}, $s_G:G\times G\to[0,1],$ 
\begin{align*}
s_G(x,y)=\frac{|x-y|}{\inf_{z\in\partial G}(|x-z|+|z-y|)}, 
\end{align*}
the \emph{point pair function} \cite[p. 685]{chkv}, \cite[2.4, p. 1124]{hvz}, $p_G:G\times G\to[0,1],$
\begin{align*}
p_G(x,y)=\frac{|x-y|}{\sqrt{|x-y|^2+4d_G(x)d_G(y)}},   
\end{align*}
the \emph{$w$-quasi-metric}, introduced in \cite{fss}, $w_G:G\times G\to[0,1]$,
\begin{align*}
&w_G(x,y)=\frac{|x-y|}{\min\{\inf_{\widetilde{y}\in\widetilde{Y}}|x-\widetilde{y}|,\inf_{\widetilde{x}\in\widetilde{X}}|y-\widetilde{x}|\}}\quad\text{with}\\
&\widetilde{X}=\{\widetilde{x}\in S^{n-1}(x,2d_G(x))\text{ }|\text{ }(x+\widetilde{x})\slash2\in\partial G\},
\end{align*} 
the \emph{$t$-metric}, introduced in \cite{inm}, $t_G:G\times G\to[0,1]$,
\begin{align*}
t_G(x,y)=\frac{|x-y|}{|x-y|+d_G(x)+d_G(y)},  
\end{align*}
and the \emph{Barrlund metric} \cite{fmv}, first studied in \cite{b99} and \cite{rcli}, $b_{G,p}:G\times G\to[0,\infty)$,
\begin{align*}
b_{G,p}(x,y)=\sup_{z\in\partial G}\frac{|x-y|}{(|x-z|^p+|z-y|^p)^{1\slash p}}\quad\text{for some}\quad p\geq1.
\end{align*}
Note that there are some domains in which the functions $p_G$ and $w_G$ are not metrics, as noted in \cite[Rmk 3.1 p. 689]{chkv} and \cite[Ex. 4.5, p. 9]{fss}. However, they are always at least quasi-metrics in a domain $G$ if they are defined: The function $w_G$ can only be defined in convex domains $G$ \cite[p. 8]{fss}. Furthermore, it follows from \cite[Thm 3.15 p. 11]{fmv} that, for all $x,y\in\B^n$, 
\begin{align*}
b_{\B^2,2}(x,y)=\frac{|x-y|}{\sqrt{2+|x|^2+|y|^2-2|x+y|}} 
\end{align*}
and, by \cite[Cor. 5.5, p. 12]{fss}, for all distinct points $x,y\in\B^n$ such that $0\leq|y|\leq|x|<1$ and $x\neq0$
\begin{align*}
w_{\B^n}(x,y)=\frac{|x-y|}{|y-\widetilde{x}|},
\end{align*}
where $\widetilde{x}=x(2-|x|)\slash|x|$.

Consider now a few known inequalities between these intrinsic metrics.

\begin{lemma}\label{lem_5ineq}
\cite[Lemma 2.1, p. 1124; Lemma 2.2, p. 1125 \& Lemma 2.3, p. 1125]{hvz}, \cite[Thm 3.6, p. 4]{sqm}, \cite[Thm 3.8, p. 7]{inm}, \cite[Thm 3.6 p. 7]{fmv}\newline
For a domain $G\subsetneq\R^n$ and for all $x,y\in G$, the following inequalities hold:\newline
$(1)$  $j^*_G(x,y)\leq s_G(x,y)\leq2j^*_G(x,y)$,\newline
$(2)$ $j^*_G(x,y)\leq p_G(x,y)\leq\sqrt{2}j^*_G(x,y)$,\newline
$(3)$  $(1\slash\sqrt{2})p_G(x,y)\leq s_G(x,y)\leq\sqrt{2}p_G(x,y)$,\newline
$(4)$  $(1\slash2)\max\{s_G(x,y),p_G(x,y)\}\leq t_G(x,y)\leq j^*_G(x,y)$,\newline
$(5)$  $s_G(x,y)\leq b_{G,p}(x,y)\leq2^{1-1\slash p}s_G(x,y).$
\end{lemma}

\begin{theorem}\cite[Cor. 4.9, p. 10]{fss}, \cite[Lemma 2.3, p. 1125]{hvz}
For any convex domain $G\subsetneq\R^n$ and all $x,y\in G$,
\begin{align*}
j^*_G(x,y)\leq w_G(x,y)\leq s_G(x,y)\leq p_G(x,y)\leq\sqrt{2}j^*_G(x,y).  \end{align*}
\end{theorem}

\begin{lemma}\label{lem_jwsprho}\cite[Cor. 4.10, p. 10]{fss}, \cite[Thm 3.8(1), p. 7 \& Thm 3.11, p. 9]{inm}
For all $x,y\in G\in\{\uhp^n,\B^n\}$,
\begin{align*}
&(1)\quad{\rm th}\frac{\rho_{\uhp^n}(x,y)}{4}\leq j^*_{\uhp^n}(x,y)\leq w_{\uhp^n}(x,y)=s_{\uhp^n}(x,y)= p_{\uhp^n}(x,y)={\rm th}\frac{\rho_{\uhp^n}(x,y)}{2},\\
&(2)\quad{\rm th}\frac{\rho_{\B^n}(x,y)}{4}\leq j^*_{\B^n}(x,y)\leq w_{\B^n}(x,y)\leq s_{\B^n}(x,y)\leq p_{\B^n}(x,y)\leq{\rm th}\frac{\rho_{\B^n}(x,y)}{2},\\
&(3)\quad\frac{1}{2}{\rm th}\frac{\rho_G(x,y)}{2}\leq t_G(x,y)\leq j^*_G(x,y)\leq{\rm th}\frac{\rho_G(x,y)}{2}.
\end{align*}
\end{lemma}

\begin{lemma}\cite[Lemma 3.5, p. 5 \& Thm 3.7 p. 6]{sinb}
For all $x,y\in\B^n$,
\begin{align*}
\max\left\{s_{\B^n}(x,y),\frac{4}{\sqrt{10}+\sqrt{2}}p_{\B^n}(x,y)\right\}
\leq b_{\B^2,2}(x,y)
\leq\sqrt{2}s_{\B^n}(x,y).
\end{align*}
\end{lemma}

\begin{remark}\label{rmk_th}
For every $t\geq0$,
\begin{align*}
{\rm th}(t)=\frac{2{\rm th}(t\slash2)}{1+{\rm th}^2(t\slash2)}\leq2{\rm th}(t\slash2).
\end{align*}
\end{remark}

\section{Inequalities for metrics}\label{s3}

In this section, we will find inequalities between different intrinsic metrics and the hyperbolic metric defined in the unit ball $\B^n$ that work for all points $x,y\in\B^n$ whose absolute values are on a certain interval $[r_l,r_u]\subset[0,1)$. In other words, we improve the bounds of Lemma \ref{lem_jwsprho}(2)-(3) by taking the absolute values of the points into account. Furthermore, in Corollary \ref{cor_brhoB} and Lemma \ref{lem_brhoH}, we will also introduce such bounds for the Barrlund metric in terms of the hyperbolic metric that work for all points in the unit disk and in the upper half-space, respectively. 

\begin{remark}\label{rmk_nwlg}
When studying the inequalities between the hyperbolic metric and the intrinsic metrics introduced in this paper in the unit ball or in the upper half-space, we can always suppose that $n=2$ without loss of generality. Namely, the distances between the points $x,y\in\B^n$ defined by these metrics only depend on how $x,y$ are located on the intersection of the domain $\B^n$ and the two-dimensional plane containing $x$, $y$ and the origin. In $\uhp^n$, the distances are determined by the location of $x$ and $y$ on the intersection of $\uhp^n$ and the two-dimensional plane that contains $x,y$ and is perpendicular to $\partial\uhp^n$.
\end{remark}

First, we find the bounds for the $t$-metric, the $j^*$-metric and the point pair function.

\begin{theorem}\label{thm_rtrho}
For all $x,y\in\B^n$  such that $0\leq|x|\leq|y|\leq r_u<1$,
\begin{align*}
\frac{1}{2}{\rm th}\frac{\rho_{\B^n}(x,y)}{2}
\leq t_{\B^n}(x,y)
\leq\frac{1+r_u}{2}{\rm th}\frac{\rho_{\B^n}(x,y)}{2} 
\end{align*}
and the latter part of this inequality has the best possible constant depending only on $r_u$.
\end{theorem}
\begin{proof}
By Remark \ref{rmk_nwlg}, we can suppose that $n=2$ without loss of generality. Now,
\begin{align}\label{quo_rtrho}
\frac{t_{\B^2}(x,y)}{{\rm th}(\rho_{\B^2}(x,y)\slash2)}
=\frac{\sqrt{|x-y|^2+(1-|x|^2)(1-|y|^2)}}{|x-y|+2-|x|-|y|}.
\end{align}
Denote $v=|x-y|^2$. By differentiation,
\begin{align*}
&\frac{\partial}{\partial v}\left(\frac{\sqrt{v^2+(1-|x|^2)(1-|y|^2)}}{v+2-|x|-|y|}\right) =\frac{(2-|x|-|y|)v-(1-|x|^2)(1-|y|^2)}{(v+2-|x|-|y|)^2\sqrt{v^2+(1-|x|^2)(1-|y|^2)}}\geq0\\
&\Leftrightarrow\quad
v\geq\frac{(1-|x|^2)(1-|y|^2)}{2-|x|-|y|}
\end{align*}
and, consequently, the quotient \eqref{quo_rtrho} obtains its maximum when $v$ is either at its lowest or at its greatest. By the triangle inequality, $0\leq v\leq|x|+|y|$ and the values of the quotient \eqref{quo_rtrho} of these endpoints are
\begin{align}
&\frac{\sqrt{(1-|x|^2)(1-|y|^2)}}{2-|x|-|y|},\quad\text{and}\label{ubrt0}\\
&\frac{\sqrt{(|x|+|y|)^2+(1-|x|^2)(1-|y|^2)}}{|x|+|y|+2-|x|-|y|}
=\frac{\sqrt{1+2|x||y|+|x|^2|y|^2}}{2}
=\frac{1+|x||y|}{2}.\label{ubrt1}
\end{align}
Since
\begin{align*}
&\frac{\partial}{\partial|x|}\left(\frac{\sqrt{(1-|x|^2)(1-|y|^2)}}{2-|x|-|y|}\right)
=\frac{(1-2|x|+|x||y|)}{(2-|x|-|y|)^2}\sqrt{\frac{1-|y|^2}{1-|x|^2}}\\
&\geq\left(\frac{1-|x|}{2-|x|-|y|}\right)^2\sqrt{\frac{1-|y|^2}{1-|x|^2}}
>0
\end{align*}
the quotient \eqref{ubrt0} is increasing with respect to $|x|$. Because $|x|\leq|y|$, we will have
\begin{align*}
\frac{\sqrt{(1-|x|^2)(1-|y|^2)}}{2-|x|-|y|}
\leq\frac{(1-|y|^2)}{2(1-|y|)}=\frac{1+|y|}{2},
\end{align*}
which clearly greater than the value of the quotient \ref{ubrt1} and increasing with respect to $|y|\leq r_u$. The latter part of the inequality in the theorem follows from this and its sharpness can be seen by considering points $x\to y^-$ and $y=r_u$. The first part of the inequality follows directly from Lemma \ref{lem_jwsprho}(3).
\end{proof}

\begin{figure}[ht]
    \centering
    \begin{tikzpicture}[scale=7]
    \draw[thick] (1,0) arc (0:90:1);
    \draw[thick] (0.565,0) arc (0:90:0.565);
    \draw[thick] (0.4,0.4) circle (0.3cm);
    \draw[thick] (0.4,0.4) circle (0.01cm);
    \draw[thick] (0,0) circle (0.01cm);
    \draw[thick] (0,0) -- (0.4,0.4);
    \draw[thick] (0.105,0.46) circle (0.01cm);
    \draw[thick] (0,0) -- (0.105,0.46);
    \draw[thick] (0.105,0.46) -- (0.4,0.4);
    \draw[thick] (0.33,0.415) arc (160:221:0.07);
    \node[scale=1.3] at (0.07,0.49) {$x$};
    \node[scale=1.3] at (0.4,0.46) {$y$};
    \node[scale=1.3] at (0.29,0.37) {$k$};
    \node[scale=1.3] at (0.21,0.47) {$v$};
    \node[scale=1.3] at (-0.03,0.07) {$0$};
    \end{tikzpicture}
    \caption{Angle $k=\arg(y\slash(y-x))$, when $y\in\B^2$ and $x\in S^1(y,v)\cap\overline{B}^2(0,|y|)$, as in the proof of Theorem \ref{thm_rjrho}.}
    \label{schfig1}
\end{figure}
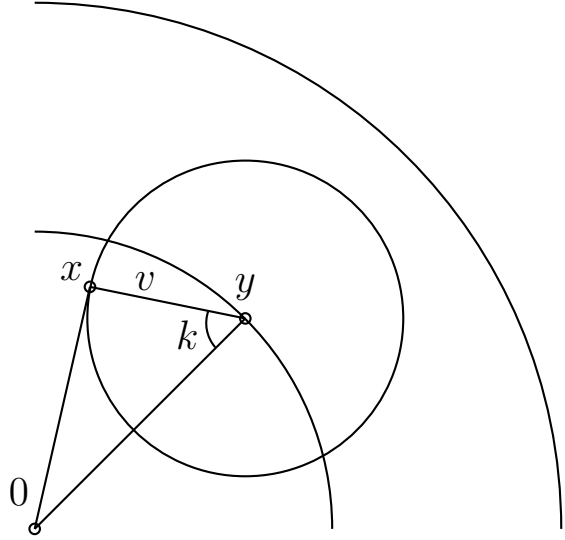

\begin{theorem}\label{thm_rjrho}
For all $x,y\in\B^n$  such that $0\leq r_l\leq|x|\leq|y|\leq r_u<1$,
\begin{align*}
&\frac{1+r_l^2}{2}{\rm th}\frac{\rho_{\B^n}(x,y)}{2}
\leq j^*_{\B^n}(x,y)
\leq\frac{1+r_u}{2}{\rm th}\frac{\rho_{\B^n}(x,y)}{2}
\\
&\text{if}\quad r_l<\frac{(48\sqrt{33}+208)^{1\slash3}}{6}-\frac{16}{3(48\sqrt{33}+208)^{1\slash3}}-\frac{1}{3}
\approx0.2955977425,\quad\text{and}
\\
&\frac{1+r_l}{\sqrt{5+2r_l+r_l^2}}{\rm th}\frac{\rho_{\B^n}(x,y)}{2}
\leq j^*_{\B^n}(x,y)
\leq\frac{1+r_u}{2}{\rm th}\frac{\rho_{\B^n}(x,y)}{2}
\quad\text{otherwise},
\end{align*}
and these inequalities have the best possible constants depending only on $r_l$ and $r_u$.
\end{theorem}
\begin{proof}
By Remark \ref{rmk_nwlg}, fix $n=2$. Now,
\begin{align*}
\frac{j^*_{\B^2}(x,y)}{{\rm th}(\rho_{\B^2}(x,y)\slash2)}
=\frac{\sqrt{|x-y|^2+(1-|x|^2)(1-|y|^2)}}{|x-y|+2(1-|y|)}.
\end{align*}
Suppose that $v=|x-y|>0$ is fixed. Clearly, $x\in S^1(y,v)\cap\overline{B}^2(0,|y|)$. The value of the smaller angle between lines $L(0,y)$ and $L(x,y)$ is $k=\arg(y\slash(y-x))$. By writing the distance $|x|$ with the law of cosines, we will have
\begin{align*}
\frac{j^*_{\B^2}(x,y)}{{\rm th}(\rho_{\B^2}(x,y)\slash2)}
=\frac{\sqrt{v^2+(1-v^2-|y|^2+2|y|v\cos(k))(1-|y|^2)}}{v+2(1-|y|)},
\end{align*}
which is clearly increasing with respect to $\cos(k)$ and thus decreasing with respect to $k$. Clearly, $k$ is at greatest when $|x|=|y|$ and at lowest when $x=y-yv\slash|y|$. It follows from this that
\begin{align}\label{bounds_rjrho}
\frac{\sqrt{v^2+(1-|y|^2)^2}}{v+2(1-|y|)}
\leq\frac{j^*_{\B^2}(x,y)}{{\rm th}(\rho_{\B^2}(x,y)\slash2)}
\leq\frac{1+|y|v-|y|^2}{v+2(1-|y|)},
\end{align}
and these bounds are sharp. By differentiation,
\begin{align}
&\frac{\partial}{\partial v}\left(\frac{\sqrt{v^2+(1-|y|^2)^2}}{v+2(1-|y|)}\right)
=\frac{(1-|y|)(2v-(1-|y|)(1+|y|)^2)}{\sqrt{v^2+(1-|y|^2)^2}(v+2(1-|y|))^2},\nonumber\\
&\frac{\partial}{\partial v}\left(\frac{1+|y|v-|y|^2}{v+2(1-|y|)}\right)
=-\left(\frac{1-|y|}{v+2(1-|y|)}\right)^2<0.\label{deri_jrl}
\end{align}
Denote $v_0=(1-|y|)(1+|y|)^2\slash2$. The first derivative above is positive if $v<v_0$, zero if $v=v_0$ and either negative or positive whenever $v>v_0$, depending on the value of $|y|$. Furthermore, by the triangle inequality $0\leq v\leq 2|y|$. Consequently, the lower bound in \eqref{bounds_rjrho} obtains its minimum value either at $v=2|y|$ or $v=v_0$, but $v=v_0$ is only possible if $v_0\leq2|y|$. If $v=2|y|$, the lower bound in \eqref{bounds_rjrho} is
\begin{align}\label{lbjr0}
\frac{\sqrt{4|y|^2+(1-|y|^2)^2}}{2|y|+2(1-|y|)}
=\frac{1+|y|^2}{2},
\end{align}
and, if $v=v_0$, this same lower bound is
\begin{align}\label{lbjr1}
\frac{\sqrt{v_0^2+(1-|y|^2)^2}}{v_0+2(1-|y|)}
=\frac{\sqrt{(1-|y|)^2(1+|y|)^4\slash4+(1-|y|^2)^2}}{(1-|y|)(1+|y|)^2\slash2+2(1-|y|)}
=\frac{1+|y|}{\sqrt{5+2|y|+|y|^2}}.
\end{align}
Since
\begin{align}\label{ine_whichjr}
\frac{1+|y|}{\sqrt{5+2|y|+|y|^2}}\leq\frac{1+|y|^2}{2}
\quad\Leftrightarrow\quad
(1-3|y|-|y|^2-|y|^3)^2\geq0,
\end{align}
the quotient \eqref{lbjr1} is always less than the quotient \eqref{lbjr0}, so the minimum value of the lower bound \eqref{bounds_rjrho} is \eqref{lbjr1} whenever $v_0\leq2|y|$ and \eqref{lbjr0} otherwise. By the cubic formula, we will have 
\begin{align*}
&v_0\leq2|y|
\quad\Leftrightarrow\quad
1-3|y|-|y|^2-|y|^3\leq0\\
&\Leftrightarrow\quad
|y|\geq\frac{(48\sqrt{33}+208)^{1\slash3}}{6}-\frac{16}{3(48\sqrt{33}+208)^{1\slash3}}-\frac{1}{3}
\approx0.2955977425.
\end{align*}
It can be shown by differentiation that both the quotients \eqref{lbjr0} and \eqref{lbjr1} are increasing with respect to $|y|\in[r_l,r_u]$, so the first parts of the inequalities in the theorem follow. Next, consider the derivative \eqref{deri_jrl}. We see from it that the upper bound in \eqref{bounds_rjrho} is always decreasing with respect to $v$ and therefore it is at greatest with $v\to0^+$, so
\begin{align*}
\frac{j^*_{\B^2}(x,y)}{{\rm th}(\rho_{\B^2}(x,y)\slash2)}
\leq\lim_{v\to0^+}\frac{1+|y|v-|y|^2}{v+2(1-|y|)}
=\frac{1+|y|}{2}.
\end{align*}
Since $(1+|y|)\slash2$ is increasing with respect to $|y|$, the rest of the theorem follows.
\end{proof}

\begin{theorem}\label{thm_rprho}
For all $x,y\in\B^n$ such that $0\leq r_l\leq|x|\leq|y|\leq r_u<1$, \begin{align*}
\frac{1+r_l}{2}{\rm th}\frac{\rho_{\B^n}(x,y)}{2}
\leq p_{\B^n}(x,y)\leq\frac{1+r_u^2}{2\sqrt{1-2r_u+2r_u^2}}{\rm th}\frac{\rho_{\B^n}(x,y)}{2}   
\end{align*}
and this inequality has the best possible constants depending only on $r_l$ and $r_u$.
\end{theorem}
\begin{proof}
Fix $n=2$ by Remark \ref{rmk_nwlg} and consider the quotient
\begin{align}\label{quo_rprho}
\frac{p_{\B^2}(x,y)}{{\rm th}(\rho_{\B^2}(x,y)\slash2)}
=\sqrt{\frac{|x-y|^2+(1-|x|^2)(1-|y|^2)}{|x-y|^2+4(1-|x|)(1-|y|)}}.
\end{align}
The expression above is increasing with respect to $u=|x-y|^2$ because, by differentiation,
\begin{align*}
\frac{\partial}{\partial u}\left(\frac{u+(1-|x|^2)(1-|y|^2)}{u+4(1-|x|)(1-|y|)}\right)
=\frac{(1-|x|)(1-|y|)(2-|x|-|y|)}{(u+4(1-|x|)(1-|y|))^2}>0.
\end{align*}
By triangle inequality, $|y|-|x|\leq u\leq|x|+|y|$ and the both equalities are possible here. Consequently, the strictest bounds only depending on $|x|$ and $|y|$ for the quotient \eqref{quo_rprho} can be written as
\begin{align}\label{bounds_rprho}
\begin{split}
&\frac{1-|x||y|}{2-|x|-|y|}
=\sqrt{\frac{1-2|x||y|+|x|^2|y|^2}{4-4|x|-4|y|+2|x||y|+|x|^2+|y|^2}}
\leq\frac{p_{\B^2}(x,y)}{{\rm th}(\rho_{\B^2}(x,y)\slash2)}\\
\leq&\sqrt{\frac{1+2|x||y|+|x|^2|y|^2}{4-4|x|-4|y|+6|x||y|+|x|^2+|y|^2}}.
\end{split}
\end{align}
By differentiation,
\begin{align*}
&\frac{\partial}{\partial|y|}\left(\frac{1-|x||y|}{2-|x|-|y|}\right)=\left(\frac{1-|x|}{2-|x|-|y|}\right)^2>0,\\
&\frac{\partial}{\partial|x|}\left(\frac{1+2|x||y|+|x|^2|y|^2}{4-4|x|-4|y|+6|x||y|+|x|^2+|y|^2}\right)\\
&=\frac{2(1+|x||y|)(1-|y|)(2-|x|-|y|^2+3|y|(1-|x|))}{(4-4|x|-4|y|+6|x||y|+|x|^2+|y|^2)^2}>0,
\end{align*}
so the lower bound in \eqref{bounds_rprho} is increasing with respect to $|y|$ and the upper bound in \eqref{bounds_rprho} increasing with respect to $|x|$. Consequently, it follows from $|x|\leq|y|$ that
\begin{align*}
\frac{1+|x|}{2}
=\frac{1-|x|^2}{2(1-|x|)}
\leq\frac{p_{\B^2}(x,y)}{{\rm th}(\rho_{\B^2}(x,y)\slash2)}
\leq\sqrt{\frac{1+2|y|^2+|y|^4}{4-8|y|+8|y|^2}}
=\frac{1+|y|^2}{2\sqrt{1-2|y|+2|y|^2}}.
\end{align*}
The both the bounds above are increasing with respect to $|x|$ and $|y|$, so the result follows.
\end{proof}

Now, let us consider the Barrlund metric.

\begin{theorem}\label{thm_rbrho}
For all $x,y\in\B^n$ such that $0\leq r_l\leq|x|\leq|y|\leq r_u<1$,
\begin{align*}
\sqrt{\frac{1+r_l^2}{2}}{\rm th}\frac{\rho_{\B^n}(x,y)}{2}
\leq b_{\B^n,2}(x,y)
\leq\frac{1+r_u}{\sqrt{2}}{\rm th}\frac{\rho_{\B^n}(x,y)}{2} 
\end{align*}
and this inequality has the best possible constants depending only on $r_l$ and $r_u$.
\end{theorem}
\begin{proof}
By Remark \ref{rmk_nwlg}, we can suppose without loss of generality that $n=2$. Now,
\begin{align*}
\frac{b_{\B^2,2}(x,y)}{{\rm th}(\rho_{\B^2}(x,y)\slash2)}
=\sqrt{\frac{|x-y|^2+(1-|x|^2)(1-|y|^2)}{2+|x|^2+|y|^2-2|x+y|}}. 
\end{align*}
Because $|x-y|^2=2|x|^2+2|y|^2-|x+y|^2$, this can be equivalently written as
\begin{align}\label{quo_rbrho}
\frac{b_{\B^2,2}(x,y)}{{\rm th}(\rho_{\B^2}(x,y)\slash2)}
=\sqrt{\frac{1+|x|^2+|y|^2+|x|^2|y|^2-|x+y|^2}{2+|x|^2+|y|^2-2|x+y|}}. \end{align}
Denote
\begin{align*}
u=|x+y|,\quad
v_0=1+|x|^2+|y|^2+|x|^2|y|^2,\quad
v_1=2+|x|^2+|y|^2,
\end{align*}
and note that $v_0>v_1^2\slash 4$ because
\begin{align*}
v_1^2-4v_0=-|x|^2(2-|x|^2)-|y|^2(2-|y|^2)-2|x|^2|y|^2<0.
\end{align*}
The argument inside the square root in \eqref{quo_rbrho} can be described with a function $f:[0,1)\to\R^+$, $f(u)=(v_0-u^2)\slash(v_1-2u)$, which is increasing because
\begin{align*}
f'(u)=\frac{2(u^2-v_1u+v_0)}{(v_1-2u)^2}
>\frac{2(u^2-v_1u+v_1^2\slash 4)}{(v_1-2u)^2}
=\frac{2(u-v_1\slash 2)^2}{(v_1-2u)^2}
=\sqrt{2}>0.
\end{align*}
By the triangle inequality, $|y|-|x|\leq u\leq|x|+|y|$, where both of the equalities are clearly possible, and thus the strictest possible bounds for the quotient \eqref{quo_rbrho} with fixed choice of $|x|,|y|$ can be written as
\begin{align}\label{bounds_rbrho}
\sqrt{\frac{1+2|x||y|+|x|^2|y|^2}{2+2|x|-2|y|+|x|^2+|y|^2}} 
\leq\frac{b_{\B^2,2}(x,y)}{{\rm th}(\rho_{\B^2}(x,y)\slash2)}
\leq\sqrt{\frac{1-2|x||y|+|x|^2|y|^2}{2-2|x|-2|y|+|x|^2+|y|^2}}.
\end{align}
By differentiation,
\begin{align*}
&\frac{\partial}{\partial |y|}\left(\frac{1+2|x||y|+|x|^2|y|^2}{2+2|x|-2|y|+|x|^2+|y|^2}\right)
=\frac{2(1+|x|)(1+|x||y|)(1-|y|-|x|+|x|^2)}{(2+2|x|-2|y|+|x|^2+|y|^2)^2}>0,\\
&\frac{\partial}{\partial |x|}\left(\frac{1-2|x||y|+|x|^2|y|^2}{2-2|x|-2|y|+|x|^2+|y|^2}\right)
=\frac{2(1-|y|)(1-|x||y|)(1-|x|-|y|(1-|y|))}{(2-2|x|-2|y|+|x|^2+|y|^2)^2}>0,
\end{align*}
so the lower bound in \eqref{bounds_rbrho} is increasing with respect to $|y|$ and the upper bound in \eqref{bounds_rbrho} is increasing with respect to $|x|$. Since $|x|\leq|y|$, it follows that
\begin{align*}
\sqrt{\frac{1+|x|^2}{2}}
\leq\frac{b_{\B^2,2}(x,y)}{{\rm th}(\rho_{\B^2}(x,y)\slash2)}
\leq\frac{1+|y|}{\sqrt{2}},
\end{align*}
from which the result follows.
\end{proof}

\begin{corollary}\label{cor_brhoB}
For all points $x,y\in\B^n$,
\begin{align*}
\frac{1}{\sqrt{2}}{\rm th}\frac{\rho_{\B^n}(x,y)}{2}
\leq b_{\B^n,2}(x,y)
\leq\sqrt{2}{\rm th}\frac{\rho_{\B^n}(x,y)}{2} 
\end{align*}
and the constants here are the best ones possible.
\end{corollary}
\begin{proof}
The bounds follow from Theorem \ref{thm_rbrho} by choosing $r_l=0$ and $r_u\to1^-$, and these limits are sharp because the inequality in Theorem \ref{thm_rbrho} is.  
\end{proof}

\begin{lemma}\label{lem_brhoH}
For all points $x,y\in\uhp^n$,
\begin{align*}
{\rm th}\frac{\rho_{\uhp^n}(x,y)}{2}
\leq b_{\uhp^n,2}(x,y)
\leq\sqrt{2}{\rm th}\frac{\rho_{\uhp^n}(x,y)}{2} 
\end{align*}
and the constants here are the best ones possible.
\end{lemma}
\begin{proof}
By Remark \ref{rmk_nwlg}, we can fix $n=2$. From \cite[Thm 3.12, p. 11]{fmv}, we see that
\begin{align*}
b_{\uhp^2,2}(x,y)=\frac{\sqrt{2}|x-y|}{\sqrt{|x-y|^2+{\rm Im}(x+y)^2}},
\end{align*}
so we will have
\begin{align*}
\frac{b_{\uhp^2,2}(x,y)}{{\rm th}(\rho_{\uhp^2}(x,y)\slash2)}
&=\frac{\sqrt{2}|x-\overline{y}|}{\sqrt{|x-y|^2+{\rm Im}(x+y)^2}}\\
&=\sqrt{2\cdot\frac{{\rm Re}(x-y)^2+{\rm Im}(x+y)^2}{{\rm Re}(x-y)^2+{\rm Im}(x-y)^2+{\rm Im}(x+y)^2}}.
\end{align*}
Clearly, the quotient above obtains its minimum value 1 when ${\rm Re}(x)={\rm Re}(y)$ and ${\rm Im}(y)\to0^+$, and its maximum value $\sqrt{2}$ when ${\rm Im}(x)={\rm Im}(y)$, from which the result follows.
\end{proof}

The bounds for the $w$-quasi-metric and the triangular ratio metric can be created by using the inequalities found already earlier.

\begin{corollary}\label{cor_rswrho}
For all $x,y\in B^n(0,r)$  such that $0\leq r_l\leq|x|\leq|y|\leq r_u<1$,
\begin{align*}
\frac{1+r_l}{\sqrt{5+2r_l+r_l^2}}{\rm th}\frac{\rho_{\B^n}(x,y)}{2}
\leq w_{\B^n}(x,y)
\leq s_{\B^n}(x,y)
\leq\frac{1+r_u^2}{2\sqrt{1-2r_u+2r_u^2}}{\rm th}\frac{\rho_{\B^n}(x,y)}{2}. 
\end{align*}
\end{corollary}
\begin{proof}
Follows from Lemma \ref{lem_jwsprho}(2), Theorems \ref{thm_rjrho} and \ref{thm_rprho}, and the inequality \eqref{ine_whichjr}.
\end{proof}

\section{Conformal mappings}\label{s4}

In this section, we will study the distortion of the intrinsic metric under conformal mappings. Our main results are Theorem \ref{thm_rpconf} and Corollary \ref{cor_confhypmidrot}, out of which the former gives upper and lower bounds for the distortion of six different intrinsic metrics defined in the unit disk, and the latter deals with the triangular ratio metric only. However, before moving on to these results, let us first define what a conformal mapping actually is. 

\begin{nonsec}{\bf Conformal mappings.} \cite[Def. 3.1, p. 25]{hkvbook}
Suppose that $G,G'$ are domains in $\R^n$ and let $f:G\to G'$ be a homeomorphism between them. In other words, choose $f$ so that it is continuous and has a continuous inverse function. Now, the function $f$ is \emph{conformal} if (1) its derivative $f'$ exists and is continuous, (2) its Jacobian determinant $J_f(x)$ is non-zero at every point $x\in G$, and (3) $|f'(x)h|=|f'(x)||h|$ for all $x\in G$ and $h\in\R^n$.
\end{nonsec}

By \cite[Def. 3.7, p. 27]{hkvbook}, a continuous homeomorphism $f$ between some domains $G,G'\subset\R^n$ is \emph{sense-preserving} if $J_f(x)>0$ for all points $x\in G$ and \emph{sense-reversing} if $J_f(x)<0$ for all points $x\in G$ instead. It follows from the definition above that conformal mappings are either sense-preserving or sense-reversing. In the two-dimensional case, sense-preserving conformal mappings preserve both magnitude and orientation of an angle between any curves, while sense-reversing conformal mappings only preserve the magnitude but not the orientation. A translation, a rotation and a strecthing by a factor $r>0$ are examples of sense-preserving conformal mappings, while reflections and inversions in the circle are conformal mappings that do not preserve orientation. Another well-known type of conformal mappings is introduced below. 

\begin{nonsec}{\bf M\"obius transformations.} \cite[Ex. 3.2, pp. 25-26 \& Def. 3.6, p. 27]{hkvbook}, \cite[Def., p. 298]{geo}
Denote $\overline{\R}^n=\R^n\cup\{\infty\}$ and let $\cdot$ be the symbol of the dot product. For any $t\in\R$ and $u\in\R^n\backslash\{0\}$,
\begin{align*}
P(u,t)=\{x\in\R^n\text{ }|\text{ }x\cdot u=t\}\cup\{\infty\}   
\end{align*}
is the hyperplane perpendicular at vector $u$, containing the point $t\slash|u|$. The reflection in this hyperplane is defined by the function $h_r:\overline{\R}^n\to\overline{\R}^n$,
\begin{align*}
h_r(x)=x-2(x\cdot u-t)\frac{u}{|u|^2},\quad h_r(\infty)=\infty.    
\end{align*}
Similarly, the inversion in the sphere $S^{n-1}(v,r)$ is $h_i:\overline{\R}^n\to\overline{\R}^n$,
\begin{align*}
h_i(x)=v+\frac{r^2(x-v)}{|x-v|^2},\quad h_i(v)=\infty,\quad h_i(\infty)=v.    
\end{align*}
A \emph{M\"obius transformation} is a function $f:\overline{\R}^n\to\overline{\R}^n$, $f=h_1\circ\cdots\circ h_m$, where each $h_j$ is either a reflection in some hyperplane or an inversion in a sphere, and $m\geq1$ is an integer. In the extended complex plane $\overline{\C}=\C\cup\{\infty\}$, the expression of a sense-preserving M\"obius transformation $f:\overline{\C}\to\overline{\C}$ can be written as
\begin{align}\label{exp_mob}
f(z)=\frac{sz+t}{uz+v},    
\end{align}
where $s,t,u,v\in\C$ are constants such that $sv-tu\neq0$. 
\end{nonsec}

\begin{remark}
(1) In the extended complex plane, the expression for the sense-reversing M\"obius transformation can be obtained by replacing the variable $z$ by its complex conjugate $\overline{z}$ in the expression \eqref{exp_mob}.\newline
(2) It follows from the Schwarz lemma that a conformal mapping $f:G\to G'=f(G)$, $G,G'\in\{\B^2,\uhp^2\}$ is in fact a M\"obius transformation.\newline
(3) If $G$ is a domain in $G\subset\R^n$, $n\geq3$, and $f:G\to f(G)\subset\R^n$ is a conformal mapping, then by \emph{Liouville's theorem} there is a M\"obius transformation $h$ in $\overline{\R}^n=\R^n\cup\{\infty\}$ such that $h(x)=f(x)$ for every $x\in G$ \cite[Rmk 3.44, p. 47]{hkvbook}.
\end{remark}

Let us now move on to study the distortion of the intrinsic metrics under conformal mappings. Out of the different metrics and quasi-metrics studied in this paper, only the hyperbolic metric is conformally invariant but it is clear that conformal mappings cannot distort the distances of the other metrics indefinitely because these metrics have bounds in terms of the hyperbolic metric. If the domain is the upper half-space or the unit ball, the upper bound of the distortion is 2, as shown below.

\begin{lemma}\label{lem_confG}
Suppose that $G,G'\in\{\uhp^n,\B^n\}$, $f:G\to G'=f(G)$ is a conformal mapping and $d_G$ is one of the intrinsic metrics or quasi-metrics in $\{t_G,j^*_G,w_G,s_G,p_G\}$. Then, for all $x,y\in G$,
\begin{align*}
d_G(x,y)\slash2\leq d_{G'}(f(x),f(y))\leq2d_G(x,y).    
\end{align*}
Furthermore, for $d_G\in\{w_G,s_G,p_G\}$,
\begin{align*}
&d_G(x,y)\slash2\leq d_{G'}(f(x),f(y))\leq d_G(x,y)\quad\text{if}\quad G=\uhp^n,\,G'=\B^n,\quad\text{and}\\
&d_G(x,y)\leq d_{G'}(f(x),f(y))\leq2d_G(x,y)\quad\text{if}\quad G=\B^n,\,G'=\uhp^n.
\end{align*}
\end{lemma}
\begin{proof}
From Lemma \ref{lem_jwsprho} and Remark \ref{rmk_th}, we see that
\begin{align*}
\frac{1}{2}{\rm th}\frac{\rho_G(x,y)}{2}\leq d_G(x,y)\leq{\rm th}\frac{\rho_G(x,y)}{2}   
\end{align*}
for all $d_G\in\{t_G,j^*_G,w_G,s_G,p_G\}$. Thus, by the conformal invariance of the hyperbolic metric, 
\begin{align*}
\frac{1}{2}d_G(x,y)
&\leq\frac{1}{2}{\rm th}\frac{\rho_G(x,y)}{2}
=\frac{1}{2}{\rm th}\frac{\rho_{G'}(f(x),f(y))}{2}
\leq d_{G'}(f(x),f(y))
\leq{\rm th}\frac{\rho_{G'}(f(x),f(y))}{2}\\
&={\rm th}\frac{\rho_G(x,y)}{2}
\leq2d_G(x,y).
\end{align*}
By Lemma \ref{lem_jwsprho}(1), $d_{\uhp^n}(x,y)={\rm th}(\rho_{\uhp^n}(x,y)\slash2)$ for $d_G\in\{w_G,s_G,p_G\}$, from which the latter part of the result follows.
\end{proof}

\begin{lemma}
For any conformal mapping $f:\uhp^n\to\uhp^n$ and all points $x,y\in\uhp^n$,
\begin{align*}
\frac{j^*_{\uhp^n}(f(x),f(y))}{j^*_{\uhp^n}(x,y)}\leq\sqrt{2}.
\end{align*}
\end{lemma}
\begin{proof}
Follows from Lemmas \ref{lem_jwsprho}(1) and \ref{lem_5ineq}(2), and the conformal invariance of the hyperbolic metric.
\end{proof}

\begin{lemma}\label{lem_bconf}
For any conformal mapping $f:G\to G'=f(G)$ such that $G,G'\in\{\uhp^n,\B^n\}$ and for all $x,y\in G$,
\begin{align*}
&b_{G,2}(x,y)\slash\sqrt{2}
\leq b_{G',2}(f(x),f(y))\leq\sqrt{2}b_{G,2}(x,y)\quad\text{if}\quad G=G'=\uhp^n,\\
&b_{G,2}(x,y)\slash2
\leq b_{G',2}(f(x),f(y))\leq\sqrt{2}b_{G,2}(x,y)\quad\text{if}\quad G=\uhp^n,\,G'=\B^n,\\
&b_{G,2}(x,y)\slash\sqrt{2}
\leq b_{G',2}(f(x),f(y))\leq2b_{G,2}(x,y)\quad\text{if}\quad G=\B^n,\,G'=\uhp^n,\\
&b_{G,2}(x,y)\slash2
\leq b_{G',2}(f(x),f(y))\leq2b_{G,2}(x,y)\quad\text{if}\quad G=G'=\B^n.
\end{align*}
\end{lemma}
\begin{proof}
Follows from Corollary \ref{cor_brhoB}, Lemma \ref{lem_brhoH} and the conformal invariance of the hyperbolic metric.
\end{proof}

By choosing some bounds for the absolute value of the points in the unit ball, we can find a better upper and lower bounds for the distortion under conformal mappings than $1\slash2$ and $2$.

\begin{theorem}\label{thm_rpconf}
If $f:\B^n\to\B^n=f(\B^n)$ is a conformal mapping, $x,y\in\B^n$ and $r_l,r_u,R_l,R_u\in[0,1)$ such that $|x|,|y|\in[r_l,r_u]$ and $|f(x)|,|f(y)|\in[R_l,R_u]$, then
\begin{align*}
(1)\quad&
\frac{1}{1+r_u}
\leq\frac{t_{\B^n}(f(x),f(y))}{t_{\B^n}(x,y)}
\leq1+R_u,\\
(2)\quad&
\frac{2(1+R_l)}{(1+r_u)\sqrt{5+2R_l+R_l^2}}
\leq\frac{j^*_{\B^n}(f(x),f(y))}{j^*_{\B^n}(x,y)}
\leq\frac{(1+R_u)\sqrt{5+2r_l+r_l^2}}{2(1+r_l)},\\
(3)\quad&
\frac{\sqrt{1-2r_u+2r_u^2}(1+R_l)}{1+r_u^2}
\leq\frac{p_{\B^n}(f(x),f(y))}{p_{\B^n}(x,y)}
\leq\frac{1+R_u^2}{(1+r_l)\sqrt{1-2R_u+2R_u^2}},\\ 
(4)\quad&
\frac{\sqrt{1+R_l^2}}{1+r_u}
\leq\frac{b_{\B^n,2}(f(x),f(y))}{b_{\B^n,2}(x,y)}
\leq\frac{1+R_u}{\sqrt{1+r_l^2}},\\
(5)\quad&
\frac{2(1+R_l)\sqrt{1-2r_u+2r_u^2}}{(1+r_u^2)\sqrt{5+2R_l+R_l^2}}
\leq\frac{s_{\B^n}(f(x),f(y))}{s_{\B^n}(x,y)}
\leq\frac{(1+R_u^2)\sqrt{5+2r_l+r_l^2}}{2(1+r_l)\sqrt{1-2R_u+2R_u^2}},
\\
(6)\quad&
\frac{2(1+R_l)\sqrt{1-2r_u+2r_u^2}}{(1+r_u^2)\sqrt{5+2R_l+R_l^2}}
\leq\frac{w_{\B^n}(f(x),f(y))}{w_{\B^n}(x,y)}
\leq\frac{(1+R_u^2)\sqrt{5+2r_l+r_l^2}}{2(1+r_l)\sqrt{1-2R_u+2R_u^2}}.
\end{align*}
\end{theorem}
\begin{proof}
Follows from Theorems \ref{thm_rtrho}, \ref{thm_rjrho}, \ref{bounds_rprho} and \ref{thm_rbrho}, the inequality \eqref{ine_whichjr}, Corollary \ref{cor_rswrho} and the conformal invariance of the hyperbolic metric.
\end{proof}

\begin{remark}
By choosing $r_l=R_l=0$ and $r_u=R_u\to1^-$, the inequalities (1), (3) and (4) in Theorem \ref{thm_rpconf} are the same as the ones in Lemmas \ref{lem_confG} and \ref{lem_bconf}(2). While this not true for the other three inequalities in Theorem \ref{thm_rpconf}, we could improve their results for certain choices of $r_l$ and $R_l$: Recall that the bound $(1+r_l)\slash\sqrt{5+2r_l+r_l^2}$ in Theorem \ref{thm_rjrho} can be replaced with $(1+r_l^2)\slash2$ in the case $r_l<0.295$. Thus, by taking into consideration whether $r_l,R_l$ are less than $0.295$ or not, the other inequalities would also give sharp results in the case $r_l=R_l=0$ and $r_u=R_u\to1^-$. 
\end{remark}

Next, consider the sense-preserving M\"obius transformation
\begin{align}\label{exp_Ta}
T_a:\B^2\to\B^2,\quad T_a(z)=\frac{z-a}{1-\overline{a}z},    
\end{align}
where $a$ is some fixed point from the unit disk. This hyperbolic isometry has several useful properties that can be found in \cite[Ch. 3.2, pp. 35-38]{hkvbook}. Interestingly, it seems to be fulfilling the inequality
\begin{align}\label{ine_Ta}
d_{\B^2}(T_a(x),T_a(y))\leq(1+|a|)d_{\B^2}(x,y),\quad x,y\in\B^2
\end{align}
for several intrinsic metrics, including the triangular ratio metric \cite[Conj. 1.6, p. 684]{chkv}, the Barrlund metric with $p=2$ \cite[Conj. 4.3, p. 25]{fmv} and the $t$-metric \cite[Conj. 4.13, p. 13]{inm}. Numerical tests suggest that this inequality would also hold for the $j^*$-metric, the point pair function and the quasi-metric $w$. Furthermore, the inequality \eqref{ine_Ta} has the best possible constant as can be see in the next lemma.

\begin{lemma}\label{lem_Tasup}
For all $a\in\B^2$ and $d_G\in\{t_G,j^*_G,w_G,s_G,p_G,b_{G,2}\}$,
\begin{align*}
\sup_{x,y\in\B^2}\frac{d_{\B^2}(T_a(x),T_a(y))}{d_{\B^2}(x,y)}\geq1+|a|.\end{align*}
\end{lemma}
\begin{proof}
By fixing $\mu=\arg(a)$ and $0<k<1$, we can calculate that
\begin{align*}
\lim_{k\to0^+}\frac{d_{\B^2}(T_a(ke^{\mu i}),T_a(-ke^{\mu i}))}{d_{\B^2}(ke^{\mu i},-ke^{\mu i})}=1+|a|   
\end{align*}
for all $d_G\in\{t_G,j^*_G,w_G,s_G,p_G,b_{G,2}\}$, from which the result follows.
\end{proof}

According to several numerical tests, the equality holds in Lemma \ref{lem_Tasup} for all six intrinsic metrics and quasi-metrics considered and, trivially, the inequality \ref{ine_Ta} would follow from this. 

\begin{corollary}\label{cor_mobhforw}
For any M\"obius transformation $h:\B^n\to\B^n$, all $x,y\in\B^n$ and all $d_G\in\{j^*_G,w_G,s_G,p_G\}$,
\begin{align*}
d_{\B^n}(h(x),h(y))\leq\frac{2d_{\B^n}(x,y)}{1+d_{\B^n}(x,y)^2}\leq2d_{\B^n}(x,y),    
\end{align*}
where the constant 2 in the second inequality is sharp.
\end{corollary}
\begin{proof}
The inequality follows from Lemma \ref{lem_jwsprho}(2), Remark \ref{rmk_th} and the conformal invariance of the hyperbolic metric, and the sharpness from Lemma \ref{lem_Tasup} by letting $a\to1^-$.
\end{proof}

\begin{lemma}\label{lem_BunderTa}
For all $0<r<1$ and $a\in\B^2$, 
\begin{align*}
B^2(0,R_l)\subset T_a(B^2(0,r))\subset B^2(0,R_u),    
\end{align*}
if and only if
\begin{align*}
R_l\leq\frac{||a|-r|}{1-|a|r}\quad\text{and}\quad
R_u\geq\frac{|a|+r}{1+|a|r}.
\end{align*}
\end{lemma}
\begin{proof}
From the expression \eqref{exp_Ta}, we see that $T_0$ is clearly identity transformation, so the result holds trivially if $a=0$. Suppose then that $a\in\B^2\backslash\{0\}$. By \cite[Ex. 3.21(1), p. 37]{hkvbook},
\begin{align*}
|T_a(y)|=\frac{|a-y|}{|a||y-a\slash|a|^2|}=\frac{|a-y|}{||a|y-a\slash|a||}    
\end{align*}
for all $y\in\B^2$. If $\mu=\arg(a)-\arg(y)$, by the law of cosines,
\begin{align*}
|T_a(y)|=\sqrt{\frac{|a|^2+|y|^2-2|a||y|\cos(\mu)}{1+|a|^2|y|^2-2|a||y|\cos(\mu)}},    
\end{align*}
and, since $|a|^2+|y|^2<1+|a|^2|y|^2$, the expression above is decreasing with respect to $\cos(\mu)\in[-1,1]$. Consequently, 
\begin{align*}
\inf_{y\in S^1(0,r)}|T_a(y)|=\frac{|a|^2+r^2-2|a|r}{1+|a|^2r^2-2|a|r}
=\frac{||a|-r|}{1-|a|r},\quad
\sup_{y\in S^1(0,r)}|T_a(y)|=\frac{|a|+r}{1+|a|r},
\end{align*}
from which our result follows.
\end{proof}

\begin{corollary}\label{cor_RlRuforT}
The M\"obius transformation $f=T_a$ defined for any $a\in\B^2$ fulfills the inequalities of Theorem \ref{thm_rpconf} for all $x,y\in\B^2$ such that $0\leq r_l\leq|x|\leq|y|\leq r_u<1$ with $R_l=||a|-r_l|\slash(1-|a|r_l)$ and $R_u=(|a|+r_u)\slash(1+|a|r_u)$.
\end{corollary}
\begin{proof}
Follows from Lemma \ref{lem_BunderTa}. 
\end{proof}

Next, we will yet study the distortion of the triangular ratio metric defined on the unit disk under conformal mappings but, to able to use study these results, we need to first introduce an expression for the hyperbolic midpoint of two points in the unit disk.

\begin{theorem}\cite[Thm 1.4, p. 3]{wvz}
For all $x,y\in\B^2$, the hyperbolic midpoint $q$ of $x$ and $y$ with $\rho_{\B^2}(x,q)=\rho_{\B^2}(q,y)=\rho_{\B^2}(x,y)\slash2$ is given by
\begin{align*}
q=\frac{y(1-|x|^2)+x(1-|y|^2)}{1-|x|^2|y|^2+A[x,y]\sqrt{(1-|x|^2)(1-|y|^2)}},   
\end{align*}
where $A[x,y]=\sqrt{|x-y|^2+(1-|x|^2)(1-|y|^2)}$.
\end{theorem}

According to the results of the article \cite{sinb}, the triangular ratio distance for any two points $x,y\in\B^2$ can be estimated by rotating these points $x,y$ on the hyperbolic circle whose center is the hyperbolic midpoint $q$ of $x,y$. Namely, the value of triangular ratio metric between the points $x,y$ is at lowest when $x,y$ are rotated so that $|x|=|y|$ and at highest when the rotation is done so that $x,y$ are collinear with the origin. It follows from this that we have bounds for the triangular ratio distance $s_{\B^2}(x,y)$ that can be expressed only in terms of the absolute value of the hyperbolic midpoint of $x,y$ and the radius of the hyperbolic circle on which $x,y$ are rotated, and this result can be extended also to the general case where $n\geq2$.

\begin{theorem}\label{thm_hypmidrot}
Choose any distinct points $x,y\in\B^n$. Consider the hyperbolic sphere $S_\rho(q,t)$, whose center $q$ is the hyperbolic midpoint of $x$ and $y$, and radius $t$ is the distance ${\rm th}(\rho_{\B^n}(x,y)\slash4)$. Trivially, $x,y\in S_\rho(q,t)$ and the triangular ratio distance between these points fulfills the inequality 
\begin{align*}
\sqrt{\frac{|q|^2+t^2}{1+|q|^2t^2}}
&\leq s_{\B^n}(x,y)
\leq\frac{(1+|q|)t}{1+|q|t^2},
\quad\text{if}\quad|q|<t^2\\
\frac{t(1+|q|)}{\sqrt{(1+t^2)(1+|q|^2t^2)}}
&\leq s_{\B^n}(x,y)
\leq\frac{(1+|q|)t}{1+|q|t^2},
\quad\text{otherwise}.
\end{align*}
Furthermore, these are the best bounds for $s_{\B^n}(x,y)$ possible expressed only in terms $q$ and $t$.
\end{theorem}
\begin{proof}
Follows from \cite[Def. 5.1, p. 17; Thm 5.3, p. 20; Thm 5.11, p. 22 \& Thm 5.12, p. 23]{sinb} and Remark \ref{rmk_nwlg}.
\end{proof}

Finally, we can prove the following corollary.

\begin{corollary}\label{cor_confhypmidrot}
For any conformal mapping $f:\B^n\to f(\B^n)=\B^n$ and for all distinct points $x,y\in\B^n$ with a hyperbolic midpoint $q$ and the distance $t={\rm th}(\rho_{\B^n}(x,y)\slash4)$, 
\begin{align*}
l(|q|,t)
\leq\frac{s_{\B^n}(f(x),f(y))}{s_{\B^n}(x,y)}
\leq u(|q|,t),
\end{align*}
where the functions $l,u:[0,1]\times[0,1]\to\R^+$ are defined as
\begin{equation*}
l(|q|,t)=\frac{1+|q|t^2}{1+|q|}
\quad\text{and}\quad
u(|q|,t)=
\begin{dcases*}
\frac{2t}{1+t^2}\sqrt{\frac{1+|q|^2t^2}{|q|^2+t^2}} & if $|q|<t^2$,\\
\frac{2}{1+|q|}\sqrt{\frac{1+|q|^2t^2}{1+t^2}} & otherwise.
\end{dcases*}
\end{equation*}
\end{corollary}
\begin{proof}
Note that the points $x,y$ are on the hyperbolic sphere $S_\rho(q,t)$ and, by Theorem \ref{thm_hypmidrot}, we have a lower and upper bound for the distance $s_{\B^n}(x,y)$ expressed only in terms of $|q|$ and $t$. By Lemma \ref{lem_jwsprho}(2), Remark \ref{rmk_th} and the conformal invariance of the hyperbolic metric, 
\begin{align}\label{ine_forfinhmr}
t
={\rm th}\frac{\rho_{\B^n}(f(x),f(y))}{4}
\leq s_{\B^n}(f(x),f(y))
\leq{\rm th}\frac{\rho_{\B^n}(f(x),f(y))}{2}
=\frac{2t}{1+t^2}
\end{align}
Combining the inequality \ref{ine_forfinhmr} with the bounds for $s_{\B^n}(x,y)$ in Theorem \ref{thm_hypmidrot}, it follows that 
\begin{align*}
\frac{1+|q|t^2}{1+|q|}
\leq\frac{s_{\B^n}(f(x),f(y))}{s_{\B^n}(x,y)}
\leq\frac{2t}{1+t^2}\sqrt{\frac{1+|q|^2t^2}{|q|^2+t^2}},\quad\text{if}\quad|q|<t^2\\
\frac{1+|q|t^2}{1+|q|}
\leq\frac{s_{\B^n}(f(x),f(y))}{s_{\B^n}(x,y)}
\leq\frac{2}{1+|q|}\sqrt{\frac{1+|q|^2t^2}{1+t^2}}\quad\text{otherwise}.
\end{align*}
\end{proof}

Let us yet consider the bounds of Corollary \ref{cor_confhypmidrot}. By differentiation, it can be proved that all three quotients
\begin{align*}
\frac{1+|q|t^2}{1+|q|},\quad
\frac{2t}{1+t^2}\sqrt{\frac{1+|q|^2t^2}{|q|^2+t^2}}
\quad\text{and}\quad
\frac{2}{1+|q|}\sqrt{\frac{1+|q|^2t^2}{1+t^2}}
\end{align*}
are strictly decreasing with respect to $|q|$. Furthermore, the function $u(|q|,t)$ is continuous at $|q|=t^2$ because
\begin{align*}
u(t^2,t)=\lim_{|q|\to(t^2)^-}=\frac{2\sqrt{1+t^6}}{(1+t^2)^{3\slash2}}.  \end{align*}
Thus, it follows that, for all $0\leq|q|<1$ and $0<t<1$, 
\begin{align*}
&\frac{1}{2}<\frac{1+t^2}{2}=l(1,t)<l(|q|,t)\leq l(0,t)=1,\\
&1=u(1,t)<u(|q|,t)\leq u(0,t)=\frac{2}{1+t^2}<2.
\end{align*}
Consequently, without the information about the hyperbolic midpoint $q$ of $x$ and $y$, the inequalities in Corollary \ref{cor_confhypmidrot} can be written as just one inequality
\begin{align*}
\frac{1+t^2}{2}
\leq\frac{s_{\B^n}(f(x),f(y))}{s_{\B^n}(x,y)}
\leq\frac{2}{1+t^2}
\end{align*}
where $t={\rm th}(\rho_{\B^n}(x,y)\slash4)$. This result would also follow directly from Lemma \ref{lem_jwsprho}(2), Remark \ref{rmk_th} and the conformal invariance of the hyperbolic metric. Furthermore, because $l(|q|,t)>1\slash2$ and $u(|q|,t)<2$, the inequality of Corollary \ref{cor_confhypmidrot} is always better than the result in Lemma \ref{lem_confG} for the triangular ratio metric.


Let us yet consider the results of Theorem \ref{thm_rpconf}(5) and Corollary \ref{cor_confhypmidrot} through one example.  

\begin{example}\label{ex_boundComp}
Let us fix $a=0.7$, $x=0.1+0.3i$ and $y=0.3+0.5i$. We can compute that \begin{align}\label{quo_Tinex}
\frac{s_{\B^2}(T_a(x),T_a(y))}{s_{\B^2}(x,y)}\approx1.162104,    
\end{align}
where the mapping $T_a$ is as in \eqref{exp_Ta}. Let us then fix $r_l=|x|$, $r_u=|y|$ and choose $R_l$, $R_u$ as in Corollary \ref{cor_RlRuforT}. Now, by applying Theorem \ref{thm_rpconf}(5), we will obtain bounds 
\begin{align}\label{1stbounds_Tinex}
0.6399585
\leq\frac{s_{\B^2}(T_a(x),T_a(y))}{s_{\B^2}(x,y)}
\leq1.818284,
\end{align}
which clearly contain the numerical value of the quotient \eqref{quo_Tinex}. However, if we use Theorem \ref{thm_hypmidrot} to first find the hyperbolic midpoint $q$ of $x,y$, Corollary \ref{cor_confhypmidrot} gives as the bounds
\begin{align*}
0.6964436
\leq\frac{s_{\B^2}(T_a(x),T_a(y))}{s_{\B^2}(x,y)}
\leq1.356354,
\end{align*}
which are better bounds for the quotient \eqref{quo_Tinex} than the ones in \eqref{1stbounds_Tinex}.
\end{example}

In Example \ref{ex_boundComp}, we noticed that, for certain choices of $a,x,y\in\B^2$, the bounds given by Corollary \ref{cor_confhypmidrot} were better for estimating the quotient $s_{\B^2}(T_a(x),T_a(y))\slash s_{\B^2}(x,y)$ than the bounds found by applying Theorem \ref{thm_rpconf}(5) with $r_l=|x|$, $r_u=|y|$ and $R_l,R_u$ as in Corollary \ref{cor_RlRuforT}. The question whether this result holds more commonly or not can be studied with computer experiments: One just needs to choose an arbitrary point $a\in\B^2$ and two other points from $\B^2$, name these two points to $x,y$ so that $|x|\leq|y|$ and compute the bounds as in Example \ref{ex_boundComp}, and repeat this process long enough. Out of 1,000,000 computer simulations like this, both of the lower and upper bound of Corollary \ref{cor_confhypmidrot} were better in 944,821 cases compared to the bounds found by applying Theorem \ref{thm_rpconf}(5) with Corollary \ref{cor_RlRuforT}. Thus, we can conclude that Corollary \ref{cor_confhypmidrot} works very often better than Theorem \ref{thm_rpconf}(5), but not always.

\begin{remark}
In the experiment above, the arbitrary points were chosen by simulating the uniform distribution on the unit disk. To do this, one must first choose the real and imaginary coordinates for a point $x$ as observations from the uniform distribution $U(-1,1)$, then accept this point if $|x|<1$, and repeat until there are enough points. This is because the points generated  by choosing their radius from $U(0,1)$ and their argument from $U(0,2\pi)$ are not uniformly distributed on the unit disk.
\end{remark}



\section{Schwarz lemma for quasiregular mappings}\label{s5}

In this section, we will yet briefly consider $K$-quasiregular and $K$-quasiconformal mappings by using the intrinsic metrics introduced earlier. The behaviour of the triangular ratio metric and other hyperbolic type metrics under quasiregular and quasiconformal mappings has been researched in some earlier works, see for instance \cite[Thms 1.2 \& 1.3, p. 684]{chkv}. Results of this kind have been studied by many authors and there are still many open problems in this field, see \cite[pp. 318-320]{hkvbook}. Below, we briefly introduce the definitions needed but, for more details about the quasiregular mappings, the reader is referred to \cite[Ch. 15, pp.  281-298]{hkvbook}.

\begin{nonsec}{\bf $K$-quasiregular mappings.} \cite[pp. 289-288]{hkvbook}
Let $G\subset\R^n$ be a domain and suppose that a function $f:G\to\R^n$ is ${\rm ACL}^n$, see definition for this from \cite[p. 150]{hkvbook}. Denote the Jacobian determinant of $f$ at point $x\in G$ by $J_f(x)$. The function $f$ is \emph{quasiregular}, if there is a constant $K\geq1$ such that $f$ fulfills
\begin{align}\label{f_outerdil}
|f'(x)|^n\leq KJ_f(x),\quad
|f'(x)|=\max_{|h|=1}|f'(x)h|
\end{align}
a.e in $G$. Suppose that there is also some constant $K\geq1$ such that the inequality
\begin{align}\label{f_innerdil}
J_f(x)\leq K\ell(f'(x))^n,\quad
\ell(f'(x))=\min_{|h|=1}|f'(x)h|,
\end{align}
too, holds a.e in $G$. The \emph{outer dilatation} of $f$, denoted by $K_O(f)$, is the smallest constant $K\geq1$ for which the inequality \eqref{f_outerdil} is true and, similarly, the \emph{inner dilatation} of $f$, denoted by $K_I(f)$, is the smallest constant $K\geq1$ such that the inequality \eqref{f_innerdil} holds. The function $f$ is \emph{$K$-quasiregular}, if $\max\{K_I(f),K_O(f)\}\leq K$.
\end{nonsec}

Let $F_0,F_1$ be non-empty subsets of $\R^n$ and denote the family of all closed non-constant curves joining $F_0$ and $F_1$ in $\R^n$ by $\Delta(F_0,F_1;\R^n)$. The Gr\"otzsch capacity is the decreasing homeomorphism defined as $\gamma_n:(1,\infty)\to (0,\infty)$, 
\begin{align*}
\gamma_n(s)=\M(\Delta(\overline{\B}^n,[se_1,\infty];\R^n)),\quad s>1,
\end{align*}
where $e_1$ is the first unit vector and $\M$ stands for the conformal modulus \cite[(7.17), p. 121]{hkvbook}. For the definition and more details about the conformal modulus, see \cite[pp. 103-131]{hkvbook}. In the special case $n=2$, we have the explicit formulas \cite[(7.18), p. 122]{hkvbook}
\begin{align*}
\gamma_2(1/r)=\frac{2\pi}{\mu(r)},\quad \mu(r)=\frac{\pi}{2}\frac{\K(\sqrt{1-r^2})}{\K(r)},\quad
\K(r)=\int^1_0 \frac{dx}{\sqrt{(1-x^2)(1-r^2x^2)}}
\end{align*}
with $0<r<1$. Define then an increasing homeomorphism $\varphi_{K,n}:[0,1]\to[0,1]$, \cite[(9.13), p. 167]{hkvbook}
\begin{align*}
\varphi_{K,n}(r)=\frac{1}{\gamma_n^{-1}(K\gamma(1\slash r))},\quad
0<r<1,\,K>0.
\end{align*}
Define yet a number $\lambda_n$ by the formula \cite[(9.5) p. 157 \& (9.6), p. 158]{hkvbook} 
\begin{align*}
\log\lambda_n=\lim_{t\to\infty}((\gamma_n(t)\slash\omega_{n-1})^{1\slash(1-n)}-\log t),   
\end{align*}
where $\omega_{n-1}$ is the $(n-1)$-dimensional surface area of the unit sphere $S^{n-1}(0,1)$. Furthermore, note that $4\leq\lambda_n<2e^{n-1}$ for each $n\geq2$ and $\lambda_2=4$.

One of the most important results of the distortion theory and complex analysis in general is the Schwarz lemma. This lemma can be defined also for quasiregular mappings, see Theorem \ref{thm_schforqr}. Note that this result and much more other useful information related to the distortion theory can be found in \cite[Ch. 16, pp.  299-320]{hkvbook}.

\begin{theorem}\label{thm_schforqr}\cite[Thm 16.2, p. 300 \& Thm 16.39, p. 313]{hkvbook}
If $G,G'\in\{\uhp^n,\B^n\}$, $f:G\to f(G)\subset G'$ is a non-constant $K$-quasiregular mapping and $\alpha=K_I(f)^{1\slash(1-n)}$, then
\begin{align*}
&(1)\quad
{\rm th}\frac{\rho_{G'}(f(x),f(y))}{2}
\leq\varphi_{K,n}\left({\rm th}\frac{\rho_G(x,y)}{2}\right)
\leq\lambda_n^{1-\alpha}\left({\rm th}\frac{\rho_G(x,y)}{2}\right)^\alpha,\\
&(2)\quad
\rho_{G'}(f(x),f(y))
\leq K_I(f)(\rho_G(x,y)+\log4)\\
\intertext{holds for all $x,y\in G$. Furthermore, in the planar case $n=2$,}
&(3)\quad\rho_{G'}(f(x),f(y))\leq c(K)\max\{\rho_G(x,y),\rho_G(x,y)^{1\slash K}\} 
\end{align*}
for all $x,y\in G$, where
\begin{align*}
c(K)=2{\rm arth}(\varphi_{K,2}({\rm th}(1\slash2)))\leq v(K-1)+K, \quad v=\log(2(1+\sqrt{1-1\slash e^2}))<1.3507,    
\end{align*}
as in \cite[Thm 16.39, p. 313]{hkvbook}. Note that $c(K)\to1$ when $K\to1$ and, by the conformal invariance of the hyperbolic metric, the result $(3)$ also holds for any two simply connected planar domains $G$, $G'$ because they can be mapped conformally onto the unit disk $\B^2$.
\end{theorem}

Below in results from Theorem \ref{thm_dkqr} to Corollary \ref{cor_jpkqr},  we prove several results using part (1) of Theorem \ref{thm_schforqr} only and leave it for the interested reader to refine these results for the two-dimensional case using part (3) of Theorem \ref{thm_schforqr}.

\begin{theorem}\label{thm_dkqr}
If $f:\B^n\to f(\B^n)\subset\B^n$ is a $K$-quasiregular mapping with inner dilatation of $K_I(f)$, then for all $x,y\in\B^n$ and $\alpha\leq K_I(f)^{1\slash(1-n)}$,
\begin{align}\label{ine_qrKI}
d_{\B^n}(f(x),f(y))
\leq\varphi_{K,n}\left(\frac{2d_{\B^n}(x,y)}{1+d_{\B^n}(x,y)^2}\right)
\leq\lambda^{1-\alpha}_n\left(\frac{2d_{\B^n}(x,y)}{1+d_{\B^n}(x,y)^2}\right)^\alpha,
\end{align}
where $d_G\in\{j^*_G,w_G,s_G,p_G\}$.
\end{theorem}
\begin{proof}
If $d_{\B^n}$ is one of the metrics and quasi-metrics in $\{j^*_{\B^n},w_{\B^n},s_{\B^n},p_{\B^n}\}$, then by Lemma \ref{lem_jwsprho}(2) and Remark \ref{rmk_th},
\begin{align}\label{ine_dkqr}
d_{\B^n}(x,y)\leq
{\rm th}\frac{\rho_{\B^n}(x,y)}{2}
=\frac{2{\rm th}(\rho_{\B^n}(x,y)\slash4)}{1+{\rm th}^2(\rho_{\B^n}(x,y)\slash4)} 
\leq\frac{2d_{\B^n}(x,y)}{1+d_{\B^n}(x,y)^2}
\end{align}
for all points $x,y\in\B^n$. Thus, it follows directly from Theorem \ref{thm_schforqr}(1) that the inequality \eqref{ine_qrKI} holds with $\alpha=K_I(f)^{1\slash(1-n)}$. Note then that $\lambda_n\geq4$ and $2d_{\B^n}(x,y)\slash(1+d_{\B^n}(x,y)^2)\leq1$, so the third expression in the inequality \eqref{ine_qrKI} is decreasing with respect to $\alpha$. Consequently, $\alpha$ can be replaced by any constant smaller than $K_I(f)^{1\slash(1-n)}$. 
\end{proof}

\begin{remark}\label{rmk_K}
Because $K\geq K_I(f)$ for any $K$-quasiregular mapping $f$ with the inner dilatation of $K_I(f)$, we have $K^{1\slash(1-n)}\leq K_I(f)^{1\slash(1-n)}$, and Theorems \ref{thm_dkqr} and \ref{thm_jpqr} and Corollary \ref{cor_jpkqr} hold with $\alpha=K^{1\slash(1-n)}$.
\end{remark}

We can show that the result of Theorem \ref{thm_dkqr} is sharp, at least in the two-dimensional case.

\begin{corollary}
For all $x,y\in\B^2$, any $d_G\in\{j^*_G,w_G,s_G,p_G\}$ and every $K$-quasiregular mapping $f:\B^2\to f(\B^2)\subset\B^2$,
\begin{align*}
d_{\B^2}(f(x),f(y))
\leq\varphi_{2K,2}(d_{\B^2}(x,y)^2)
\leq4^{1-1\slash(2K)}d_{\B^2}(x,y)^{1\slash K},
\end{align*}
and the constant $4^{1-1\slash(2K)}$ here is sharp for $K=1$.
\end{corollary}
\begin{proof}
Fix $n=2$ and $\alpha=K^{1\slash(1-n)}=1\slash K$. By Remark \ref{rmk_K}, the condition $\alpha\leq K_I(f)^{1\slash(1-n)}$ of Theorem \ref{thm_dkqr} holds, and the inequality \eqref{ine_qrKI} can be now written as
\begin{align*}
d_{\B^2}(f(x),f(y))
\leq\varphi_{K,2}\left(\frac{2d_{\B^2}(x,y)}{1+d_{\B^2}(x,y)^2}\right)
\leq4^{1-1\slash(2K)}\left(\frac{d_{\B^2}(x,y)}{1+d_{\B^2}(x,y)^2}\right)^{1\slash K}.
\end{align*}
By \cite[Ex. 7.34(1), p. 125]{hkvbook}, for every $t\in[0,1)$ and $K\geq1$,
\begin{align*}
\varphi_{2,2}(t^2)=\frac{2t}{1+t^2}
\quad\text{and}\quad
\varphi_{2K,2}(t^2)=\varphi_{K,2}\left(\frac{2t}{1+t^2}\right).
\end{align*}
Thus, the inequality of our corollary follows. If $K=1$, then $4^{1-1\slash(2K)}=2$, which is sharp by Corollary \ref{cor_mobhforw}.
\end{proof}

Note that, in Theorem \ref{thm_dkqr}, the Schwarz lemma is created for each metric or quasi-metric $d_G$ in $\{j^*_G,w_G,s_G,p_G\}$. However, its proof is based on the inequality \ref{ine_dkqr}, which only requires that we have a lower bound for ${\rm th}(\rho_{\B^n}(x,y)\slash2)$ and an upper bound for ${\rm th}(\rho_{\B^n}(x,y)\slash4)$. By Lemma \ref{lem_jwsprho}(2), we know that the point pair function $p_{\B^n}$ is the best lower bound for ${\rm th}(\rho_{\B^n}(x,y)\slash2)$ and the $j^*_{\B^n}$-metric is the best upper bound for ${\rm th}(\rho_{\B^n}(x,y)\slash4)$, so it would be sensible to create a new result by using both of these metrical functions together.

\begin{theorem}\label{thm_jpqr}
If $f:\B^n\to f(\B^n)\subset\B^n$ is a $K$-quasiregular mapping with inner dilatation of $K_I(f)$, then for all $x,y\in\B^n$ and $\alpha\leq K_I(f)^{1\slash(1-n)}$,
\begin{align*}
&\frac{|f(x)-f(y)|}{|f(x)-f(y)|+2-2\max\{|f(x)|,|f(y)|\}}\\
\leq&\varphi_{K,n}\left(\frac{|x-y|\sqrt{|x-y|^2+4(1-|x|)(1-|y|)}}{|x-y|^2+2(1-|x|)(1-|y|)}\right)\\
\leq&\lambda^{1-\alpha}_n\left(\frac{|x-y|\sqrt{|x-y|^2+4(1-|x|)(1-|y|)}}{|x-y|^2+2(1-|x|)(1-|y|)}\right)^\alpha.
\end{align*}
\end{theorem}
\begin{proof}
Just like the proof of Theorem \ref{thm_dkqr}, except now we use the inequality
\begin{align*}
\frac{|x-y|}{|x-y|+2-2\max\{|x|,|y|\}}
&=j^*_{\B^n}(x,y)\leq
{\rm th}\frac{\rho_{\B^n}(x,y)}{2}
=\frac{2{\rm th}(\rho_{\B^n}(x,y)\slash4)}{1+{\rm th}^2(\rho_{\B^n}(x,y)\slash4)}\\
& \leq\frac{2p_{\B^n}(x,y)}{1+p_{\B^n}(x,y)^2}
=\frac{|x-y|\sqrt{|x-y|^2+4(1-|x|)(1-|y|)}}{|x-y|^2+2(1-|x|)(1-|y|)},
\end{align*}
which follows from \ref{lem_jwsprho}(2) and Remark \ref{rmk_th}. 
\end{proof}

\begin{corollary}\label{cor_jpkqr}
If $f:\B^n\to f(\B^n)\subset\B^n$ is a $K$-quasiregular mapping with inner dilatation of $K_I(f)$ such that $f(0)=0$, then for all $x\in\B^n$ and $\alpha\leq K_I(f)^{1\slash(1-n)}$,
\begin{align*}
\frac{|f(x)|}{2-|f(x)|}
\leq\varphi_{K,n}\left(\frac{|x|(2-|x|)}{|x|^2-2|x|+2}\right)
\leq\lambda^{1-\alpha}_n\left(\frac{|x|(2-|x|)}{|x|^2-2|x|+2}\right)^\alpha.
\end{align*}
\end{corollary}
\begin{proof}
Follows from Theorem \ref{thm_jpqr} by choosing $y=f(y)=0$.
\end{proof}

Let us yet briefly inspect how the quasi-metric $w_G$ behaves under different mappings defined in an open sector $S_\theta$. Recall that the function $w_G$ can be defined only in the case where the domain $G$ is convex, so we must let the angle of the sector be $0<\theta\leq\pi$. Our next lemma deals with one conformal mapping commonly used to change the angle of the sector $S_\theta$.

\begin{lemma}
If $0<\alpha\leq\beta\leq\pi$ and $f:S_\alpha\to S_\beta$, $f(z)=z^{(\beta\slash\alpha)}$, then the inequality
\begin{align*}
w_{S_\alpha}(x,y)\leq w_{S_\beta}(f(x),f(y))\leq\frac{\beta\sin(\alpha\slash2)}{\alpha\sin(\beta\slash2)}w_{S_\alpha}(x,y)
\end{align*}
holds for all $x,y\in S_\alpha$ and the constants here are sharp.
\end{lemma}
\begin{proof}
Follows from Lemma \cite[Lemma 4.3, p. 8]{fss} and \cite[Lemma 5.11, p. 13]{sqm}.
\end{proof}

The final lemma in this paper is about quasiconformal mappings, so let us yet briefly define them. 

\begin{nonsec}{\bf $K$-quasiconformal mappings.} \cite[Rmk 15.30 \& (15.6), p. 289]{hkvbook}, \cite[p. VI]{v1} 
If $G,G'$ are domains in $\R^n$, a homeomorphism $f:G\to G'=f(G)$ is \emph{$K$-quasiconformal} if it it fulfills the condition 
\begin{align*}
\M(\Gamma)\slash K\leq\M(f(\Gamma))\leq K\M(\Gamma)    
\end{align*}
for all curve families $\Gamma$ in $G$. If the homeomorphism $f$ is sense-preserving, it is $K$-quasiconformal if and only if it is $K$-quasiregular and injective. Consequently, the sense-preserving $K$-quasiconformal mappings form a subclass of the $K$-quasiregular mappings, but sense-reversing quasiconformal mappings are not quasiregular and non-injective quasi\-regular mappings are not quasiconformal.
\end{nonsec}

Our next result follows from \cite[Cor. 5.10(1), p. 13]{sqm}, which in turn follows from the result in Theorem \ref{thm_schforqr}(3). Note that Theorem \ref{thm_schforqr} holds trivially for all sense-preserving $K$-quasiconformal mappings, because they belong to $K$-quasiregular mappings. Furthermore, Theorem \ref{thm_schforqr} holds more generally for all $K$-quasiconformal mappings, because if a $K$-quasiconformal mappings $f$ is sense-reversing, then the function composition $f\circ\sigma$ of $f$ and any reflection $\sigma$ is a sense-preserving $K$-quasiconformal mapping and hyperbolic distances are invariant under the reflection $\sigma$.

\begin{theorem}
If $\alpha,\beta\in(0,\pi]$ and $f:S_\alpha\to S_\beta=f(S_\alpha)$ is a $K$-quasiconformal homeomorphism, the inequality 
\begin{align*}
\frac{\beta}{c(K)^K\pi\sin(\beta\slash2)}w_{S_\alpha}(x,y)^K\leq w_{S_\beta}(f(x),f(y))\leq c(K)\left(\frac{\pi}{\alpha}\sin\left(\frac{\alpha}{2}\right)\right)^{1\slash K}w_{S_\alpha}(x,y)^{1\slash K}    
\end{align*}
holds for all $x,y\in S_\alpha$ where $c(K)$ is as in Theorem \ref{thm_schforqr}.
\end{theorem}
\begin{proof}
Follows from Lemma \cite[Lemma 4.3, p. 8]{fss} and \cite[Cor. 5.10(1), p. 13]{sqm}.
\end{proof}

\def\cprime{$'$} \def\cprime{$'$} \def\cprime{$'$}
\providecommand{\bysame}{\leavevmode\hbox to3em{\hrulefill}\thinspace}
\providecommand{\MR}{\relax\ifhmode\unskip\space\fi MR }
\providecommand{\MRhref}[2]{%
  \href{http://www.ams.org/mathscinet-getitem?mr=#1}{#2}
}
\providecommand{\href}[2]{#2}


\begin{thebibliography}{10}

\bibitem{b99}{\sc  A. Barrlund,}  
The p-relative distance is a metric.
\emph{SIAM J. Matrix Anal. Appl., 21} (1999), 699–702. (electronic), doi.org/10.1137/S0895479898340883

\bibitem{bm}{\sc  A.F. Beardon and D. Minda,}  \emph{The hyperbolic metric and geometric 
function theory,} Proc.   International Workshop on  Quasiconformal Mappings and their 
Applications (IWQCMA05), eds. S. Ponnusamy, T. Sugawa and M. Vuorinen (2006), 9-56.

\bibitem{geo}{\sc  D. Brannan, M. Esplen and J. Gray,} 
\emph{Geometry.}
Cambridge University Press, 1999. 2nd Edition, 2012.

\bibitem{chkv}{\sc
J. Chen, P. Hariri, R. Kl\'en and M. Vuorinen,}
Lipschitz conditions, triangular ratio metric, and quasiconformal maps.
\emph{Ann. Acad. Sci. Fenn. Math., 40} (2015), 683-709.

\bibitem{fhmv}{\sc
M. Fujimura, P. Hariri, M. Mocanu and M. Vuorinen,}
The Ptolemy–Alhazen Problem and Spherical Mirror
Reflection.
\emph{Comput. Methods and Funct. Theory, 19} (2019), 135-155.

\bibitem{fmv}{\sc 
M. Fujimura, M. Mocanu and M. Vuorinen,} 
Barrlund's distance function and quasiconformal
maps, \emph{Complex Var. Elliptic Equ.} (to appear), doi.org/10.1080/17476933.2020.1751137, arXiv: 1903.12475.

\bibitem{gh}{\textsc{F.W. Gehring and  K. Hag},} \emph{The ubiquitous quasidisk.} With contributions
by Ole Jacob Broch. Mathematical Surveys and Monographs, 184. American Mathematical Society,
Providence, RI, 2012.

\bibitem{gp}{\sc 
F.W. Gehring and B.~P. Palka,} Quasiconformally homogeneous domains. \emph{J. Analyse Math. 30} (1976), 172--199.

\bibitem{hkvbook}{\sc
P. Hariri, R. Klén and M. Vuorinen,}
\emph{Conformally Invariant Metrics and Quasiconformal Mappings.}
Springer, 2020.

\bibitem{hvz}{\sc
P. Hariri, M. Vuorinen and X. Zhang,}
Inequalities and Bilipschitz Conditions for Triangular Ratio Metric.
\emph{Rocky Mountain J. Math., 47}, 4 (2017), 1121-1148.

\bibitem{hk}{\sc  
E. Harmaala and R. Kl\'en,} 
Ptolemy constant and uniformity. \emph{Publ. Math. Debrecen, 98}, 1-2 (2021), 15-42.

\bibitem{h}{\sc  
P. H\"ast\"o,} 
A new weighted metric, the relative metric I. \emph{J. Math. Anal. Appl., 274} (2002), 38-58.

\bibitem{imsz}{\sc  
Z. Ibragimov, M. Mohapatra, S. Sahoo and X. Zhang,} 
Geometry of the Cassinian metric and its inner metric. \emph{Bull. Malays. Math. Sci. Soc., 40} (2017), no. 1, 361-372.

\bibitem{rcli}{\sc  R.-C. Li,}  
Relative perturbation theory. I. Eigenvalue and singular value variations
\emph{SIAM J. Matrix Anal. Appl., 19} (1998), 956–982. (electronic), doi.org/10.1137/S089547989629849X

\bibitem{ps}{\sc
S. Pouliasis and A. Yu. Solynin,}
Infinitesimally small spheres and conformally invariant metrics. J. Analyse Math. (to appear).

\bibitem{fss}{\sc
O. Rainio,}
Intrinsic quasi-metrics. 
\emph{Bull. Malays. Math. Sci. Soc.} (to appear),
doi.org/10.1007/s40840-021-01089-9, Arxiv, 2011.02153.

\bibitem{inm}{\sc
O. Rainio and M. Vuorinen,}
Introducing a new intrinsic metric. Arxiv, 2010.01984. (2020).

\bibitem{sinb}{\sc
O. Rainio and M. Vuorinen,}
Triangular ratio metric in the unit disk. 
\emph{Complex Var. Elliptic Equ.} (to appear), doi.org/10.1080/17476933.2020.1870452, Arxiv, 2009.00265.

\bibitem{sqm}{\sc
O. Rainio and M. Vuorinen,}
Triangular Ratio Metric Under Quasiconformal Mappings In Sector Domains. Arxiv, 2005.11990. (2020).

\bibitem{v1}{\sc
J. V\"ais\"al\"a,} 
Lectures on $n$-dimensional quasiconformal mappings.- Lecture Notes in Math. Vol. 229, Springer- Verlag, Berlin- Heidelberg- New York, 1971.

\bibitem{wvz}{\sc
G. Wang, M. Vuorinen and X. Zhang,}
On Cyclic Quadrilaterals in Euclidean and Hyperbolic Geometries. arXiv:1908.10389. Publ. Math. Debrecen (to appear).

\end{thebibliography}
\end{document}